\documentclass[11pt]{article}
\usepackage[utf8]{inputenc}

\usepackage[bottom=4cm, left =2.5cm, right =2.5cm]{geometry} 
\usepackage{import} 
\usepackage{fancyhdr}
\usepackage{hyperref} 
\usepackage{amsthm, thmtools,mathtools} 
\usepackage[intoc, english]{nomencl} 

\usepackage{amsmath, amsfonts, amssymb} 
\usepackage{stmaryrd} 
\usepackage{colonequals} 
\usepackage{bbm} 
\usepackage{enumitem} 
\usepackage[normalem]{ulem} 

\usepackage{lmodern} 

\usepackage[dvipsnames]{xcolor}
\usepackage{pgf} 
\usepackage{graphicx} 
\usepackage{float} 

\allowdisplaybreaks[1]
\usepackage[english]{babel} 
\usepackage[T1]{fontenc} 
\usepackage{authblk} 

\usepackage{comment}
\usepackage[backend=biber,style=alphabetic,sorting=nyt,giveninits=true,maxnames=99]{biblatex}
\usepackage{csquotes}
\newtheorem{thm}{Theorem}[section]
\newtheorem{prp}[thm]{Proposition}
\newtheorem{lem}[thm]{Lemma}
\newtheorem{cor}[thm]{Corollary}

\newtheorem{df}[thm]{Definition}
\newtheorem{prp-df}[thm]{Definition-Proposition}

\theoremstyle{definition}

\theoremstyle{remark}

\newcommand{\refequ}[1]{(\ref{#1})}

\newcommand{\comments}[1]{~}
\newcommand{\E}{\mathbb{E}}
\newcommand{\condEp}[2]{\left.\E\left(#1\right\lvert #2 \right)}

\newcommand{\R}{\mathbb{R}}
\newcommand{\Proba}{\mathbb{P}}
\newcommand{\law}{\mathcal{L}}

\newcommand{\N}{\mathbb{N}}
\newcommand{\Z}{\mathbb{Z}}

\newcommand{\calF}{\mathcal{F}}

\newcommand{\sbt}{\,\begin{picture}(-1,1)(-1,-3)\circle*{3}\end{picture}\ }

\newcommand{\indic}{\mathbbm{1}}
\newcommand{\norm}[1]{\left\Vert #1\right\Vert}

\newcommand{\abs}[1]{\left\lvert #1 \right\rvert}

\newcommand{\eqdef}{\colonequals}

\newcommand\bigp[1]{\left(#1\right)}

\newcommand{\bigcro}[1]{\left[#1\right]}

\newcommand{\eps}{\varepsilon}

\renewcommand{\leq}{\leqslant}
\renewcommand{\geq}{\geqslant}

\renewcommand{\Tilde}{\widetilde}

\definecolor{DarkBlue}{rgb}{0.0, 0.28, 0.39}

\makeatletter
\@addtoreset{equation}{section}
\makeatother

\hypersetup{
    colorlinks=true,       
    linkbordercolor=DarkBlue,
    linkcolor=DarkBlue!70!DarkBlue, 
    citecolor=DarkBlue!70!DarkBlue, 
    urlcolor=DarkBlue!70!DarkBlue 
}

\AtBeginBibliography{\small}

\title{Mixing properties of a class of nonuniformly expanding maps. Application to Hölderian invariance principles.}
\date{ }
\author{V. Alouin\thanks{\'Ecole Normale Supérieure de Lyon, F-69342 Lyon, France.}, A. Bigot\thanks{LAMA, Univ Gustave Eiffel, Univ Paris Est Créteil, UMR 8050 CNRS, F-77454 Marne-La-Vallée, France.}}

\begin{document}
\maketitle

\begin{abstract}
    We study the mixing properties of a class of nonuniformly expanding maps when the return time to the basis has a weak moment of order $p>1$, up to a slowly varying function. From these computations, we deduce an invariance principle in Hölder spaces for the partial sum process of Birkhoff sums of Hölder continuous observables. The results apply to a class of intermittent maps of the unit interval. For such a map, we also prove that the Hölder invariance principle remains true for BV observables.
\end{abstract}

\section{Introduction and main results}
    \label{Section:intro}
Let $(E, d)$ be a complete bounded separable metric space with the Borel $\sigma$-algebra. We consider transformations $f$ of $E$, preserving a probability $\mu$. Suppose that $\varphi : E \to \R$ is a Hölder-continuous observable with $\int \varphi \, d\mu = 0$. It is worth noting that since $\mu$ is $f$-invariant, the sequence of increments $(\varphi \circ f^n)_n$ is stationary. We will place ourselves in the framework defined by Young (see \parencite[]{Young}) with  a basis $Y \subset E$ equipped with a \textit{reference measure} $m$ and $R : Y \to \mathbb N$ a return time to $Y$ (see Section \ref{Section:HIP} for more details). 
\\

Let $(\tau_\varphi(n))_n$ be the $\tau$-mixing coefficients associated with $(\varphi \circ f^n)_{n \geq 0}$ (see \cite{DP2004} for the definition and for properties of these coefficients). More precisely, let
\begin{align*}
    \tau_\varphi (n) = \underset{l \geq 1}{\sup} \, \frac{1}{l} \underset{1 \leq i_1<\cdots<i_l}{\sup} \tau\bigp{\sigma(\varphi(f^j) \, ; \, j \geq i_l + n), (\varphi(f^{i_1}), \cdots, \varphi(f^{i_l}))}
\end{align*}
where for any $k \geq 1$, any $Z$ taking values in $\R^k$ and any $\sigma$-algebra $\mathcal M$,
\begin{align*}
    \tau(\mathcal M, Z) 
    =
    \norm{
        \sup\left\{
            \abs{P_{Z|\mathcal M}(g) - P_Z(g)} \, : \, g \in \Lambda_1(\R^k)
        \right\}
    }_1
\end{align*}
and $\Lambda_1(\R^k)$ is the space of 1-Lipschitz functions from $\R^k$ to $\R$
with respect to the $\ell^1$ distance on $\R^k$.
Our first objective is to estimate these coefficients, following \cite{CDM2023}, when the upper tail of $R$ is of order $n^{-p}s(n)$ for $s$ a slowly varying function. 
It is known that a suitable control of these mixing coefficients leads to some deviation inequalities (see for instance \parencite[]{DP2004}, \parencite[]{CDM2023}, \parencite[]{MPR2011}) and limit theorems (see for instance \parencite[]{DP2004}, \parencite[]{CDM2023}, \parencite[]{Gir2017}) for the partial sums associated with $(\varphi \circ f^n)_{n \geq 1}$. In this article, we give an application to the invariance principle in Hölder spaces, as described in the following paragraph.
\\

It is common to study the Birkhoff sums $S_n(\varphi)$ associated with $(\varphi \circ f^n)_n$, seen as a discrete time random process on the probability space $(E, \mu)$, and the normalized Birkhoff sum process $W_n(\varphi)$, respectively defined by $S_n(\varphi) = \sum_{k =0}^{n - 1} \varphi \circ f^k$ and
\begin{align}
    \label{def:Wn_phi}
    &W_n(\varphi)(t) \eqdef n^{-1/2}\Big(S_{\lfloor nt \rfloor}(\varphi) + (nt - \lfloor nt \rfloor) \varphi \circ f^{\lfloor nt \rfloor + 1}\Big) ,
    \quad
    \, n \geq 1,\, t \in [0,1]. 
\end{align}
In the particular case when $(\varphi \circ f^n)_n$ is a sequence of i.i.d. r.v.'s with variance 1, Donsker proved in \cite{Don1951} that $\{W_n(\varphi)(t) \, : t \in [0,1]\}$ converges in distribution to a standard Brownian motion in the space of continuous functions equipped with the supremum norm. The case of dependent random variables was also intensively studied (see for instance \parencite[Chapter 6]{MPU2019} and the references therein). In this paper, one of our aim is to prove a related result that allows us to control a larger class of functionals of the Donsker process. Then, we will consider the space of functions $\mathcal H _\eta^0 [0,1]$ (see Section \ref{Section:HIP} for the definition and the properties of this space). 

In 2017, Giraudo (\parencite[]{Gir2017}) has obtained sufficient conditions in terms of $\tau$-mixing coefficients (see Theorem \ref{thm_gir}). Computing the $\tau$-mixing coefficients as described in Section \ref{Section:HIP} leads then to the following Hölderian invariance principle. 
\begin{thm}\label{hip}
    Let $f$ be a nonuniformly expanding map as defined in Section \ref{Section:HIP}, with return time $R$. Assume that there exist $p>2$, a slowly varying function at infinity $s$ that converges to 0 at infinity and a sequence $(u_n)_n$ of positive reals tending to 0 such that $e^{-nu_n} \leq m(R>n) \leq n^{-p}s(n)$. Then for any centered Hölder continuous observable $\varphi$, the series $\sigma^2(\varphi) = Var(\varphi) + 2\sum_{k \geq 1} Cov(\varphi, \varphi \circ f^k)$ converges absolutely and
    \begin{align*}
        W_n(\varphi) \to \sigma(\varphi)W \text{ in distribution in } \mathcal H _{1/2-1/p}^0 [0,1] 
    \end{align*}
    where $W$ is a standard Brownian motion. 
\end{thm}
~\\ 

Let us give some examples of transformations to which our Hölderian principle applies. In what follows we will say that a map $f$ is in the class $Hol(\gamma, \rho)$ with $\gamma > 0$ and $\rho$ a function if $f$ is given by 
    \begin{align}
        \label{def:Holland_map}
    f : 
    \left\{
          \begin{array}{ll}
            [0,1]  \longrightarrow  [0,1] \\
            x  \longmapsto  
                \left\{  
                \begin{array}{ll}
                      x(1+x^{\gamma}\rho(x)) 
                        &\text{ if } 0 \leq x \leq 1/2 
                  \\ 2x -1 
                        &\text{ if } 1/2 < x \leq 1
                \end{array}
                \right.
            \\
          \end{array}
    \right.
    \end{align}
where: 
\begin{itemize}
    \item $\rho$ is of class $C^2$ on $]0,1[$, slowly varying at $0$ (see Annex \ref{annex:slowlyvar} for some basic properties of such functions and \cite{BGT} for a thorough study),
    \item $x\rho'(x) = o(\rho(x))$ and $x^2\rho''(x) = o(\rho(x))$ at $0$,
    \item $f(1/2)= 1$ and $f'(x) > 1$  for any $x \neq 0, 1/2$, 
    \item $\rho$ converges to $0$ at $0$. 
\end{itemize}
This class of functions is a subset of the class of Holland maps defined in \cite{Hol05}, largely studied by Gouëzel in \cite{Gou2004} and \cite{GouThese}. As $0$ is a fixed point with $f'(0) = 1$, $f$ is an intermittent map for any $\gamma > 0$ and to study its degree of intermittency it is convenient to consider the basis  $Y = ]1/2,1]$ which is away from $0$. As proved in \parencite[]{Gou2004} the first return time $R$ of $f$ given by
\begin{align*}
    R(x) = \inf\{n \geq 1 : f^n(x) \in Y\}
\end{align*}
has a weak polynomial moment of order $p = 1/ \gamma$. More precisely, combining Theorem 1.4.10 and Remark 1.4.13 in \cite{Gou2004}, we infer that there exists a slowly varying function at infinity $s$ such that $m(R > n) = n^{-1/\gamma}s(n)$. Finally, if $\rho(x) \xrightarrow[x \to 0]{} 0$ then $s(x) \xrightarrow[x \to \infty]{} 0$ so that, applying our Theorem \ref{hip}, we get the following result. 
\begin{cor}
    Let $f$ be a Holland map in $Hol(\gamma, \rho)$ with $\gamma < 1/2$. Then for any centered Hölder continuous observable $\varphi$, the series $\sigma^2(\varphi) = Var(\varphi) + 2\sum_{k \geq 1} Cov(\varphi, \varphi \circ f^k)$ converges absolutely and
    \begin{align*}
        W_n(\varphi) \to \sigma(\varphi)W \text{ in distribution in } \mathcal H _{1/2-\gamma}^0 [0,1] 
    \end{align*}
    where $W$ is a standard Brownian motion. 
\end{cor}

For this type of transformations, one can naturally wonder if we can also derive the asymptotic behavior of the partial sums process associated with observables with bounded variation. As we shall see in Section \ref{Section:BV}, it is possible to suitably control some weak dependent coefficients as done in \cite{DGM2010} and then apply the Hölderian invariance principle stated in \cite{DM2017} to get the invariance principle for $\{W_n(\varphi), n \geq 1\}$ when $\varphi$ is a function with bounded variation:
\begin{thm}\label{hip_BV}
    Let $f$ be a Holland map as defined in (\ref{def:Holland_map}) where $\gamma < 1/2$.
    Then for any observable of bounded variation $\varphi$, the series $\sigma^2(\varphi) = Var(\varphi) + 2\sum_{k \geq 1} Cov(\varphi, \varphi \circ f^k)$ converges absolutely and 
     \begin{align*}
        W_n(\varphi) \to \sigma(\varphi)W \text{ in distribution in } \mathcal{H}_{1/2 - 1/p}^0[0,1]
    \end{align*}
    where $W$ is a standard Brownian motion. 
\end{thm}

        \section{Mixing properties of nonuniformly expanding maps and some applications} \label{Section:HIP}

Let us specify the precise setting we consider, analogously to Section 2 in \cite{CDM2023}. Let $(E,d, \mu)$ be a complete bounded separable metric space with the Borel $\sigma$-algebra. Assume that $f : E \to E$ is a measurable transformation with an inducing scheme as follows: 
\begin{itemize}[label = $\sbt$]
        \item a closed subset $Y \subset E$, referred to as basis, equipped with a measure $m$; 
	\item a countable (infinite) and measurable partition $\Gamma$ of $Y$ such that $m(a) > 0$ for any $a \in \Gamma$; 
	\item an integrable return time $R : Y \to \N^*$ such that $gcd\{R(y) : y \in Y\} = 1$, constant on each $a \in \Gamma$ with value $R(a)$ such that $f^{R(a)}(y) \in Y$ for any $y \in Y$. 
\end{itemize}
We consider the induced map $F_Y : Y \to Y$ defined by $F_Y(y) = f^{R(y)}(y)$ and assume that there exist $\lambda > 1$ and $K > 0$ such that for any $a \in \Gamma$ and any $x,y \in a$, 
\begin{itemize}[label = $\sbt$]
	\item $F_Y$ restricts to a (measure-theoretic) bijection from $a$ to $Y$;
	\item $d(F_Y (x), F_Y (y)) \geq \lambda d(x, y)$;
	\item $d(f^k(x), f^k(y)) \leq  Kd(F_Y (x), F_Y (y))$ for any $0 \leq k \leq R(a)$; 
	\item $\abs{\log |\zeta_a(x)| - \log |\zeta_a(y)| } \leq  Kd(F_Y (x), F_Y (y))$ where $\zeta_a = dm/ dm\circ F_Y$ is the inverse Jacobian of the restriction $F_Y : a \to Y$, 
	\item $m(Y_\omega) = m(\bar{Y_\omega})$ for any finite word $\omega = a_0 a_1 \cdots a_n \in \Gamma$, where $Y_\omega = \cap_{k = 0}^n F^{-k}(a_k)$.
\end{itemize}
We say that the map $f$ as described above is nonuniformly expanding.
For such maps it is well known that there is a unique absolutely continuous $F_Y$-invariant probability measure $\nu_Y$ on Y with $c^{-1} \leq d\nu_Y/dm \leq c$ for some $c > 0$.

\begin{prp}
        \label{prp:estim_tau_weakmoment}
    Let $f$ be a nonuniformly expanding map and $p >1$. Assume that there exists a slowly varying function at infinity $s$ and a sequence $(u_n)_n$ of positive reals tending to 0 such that $e^{-nu_n} \leq m(R>n) \leq n^{-p}s(n)$. Then for any Hölder continuous observable $\varphi$,
    \begin{align}
        \label{estim_tau_weakmoment}
        \tau_{\varphi}(n) = O(n^{1-p}s(n)). 
    \end{align}
\end{prp}

Extensive research has already been carried out to study the asymptotic behaviour of the process $\{W_n(\varphi)(t) \, : \, t \in [0,1]\}$, as defined in (\ref{def:Wn_phi}), in the space of continuous functions on $[0,1]$ equipped with $\norm{\cdot}_\infty$, to establish its convergence to a Brownian motion. 
Now, it is well known that Brownian motion's paths are almost surely Hölder regular of exponent $\eta \in ]0,1/2[$. Thus, it is natural to consider the Birkhoff sum process as an element of $\mathcal H _\eta [0,1]$ given by all the functions $x : [0,1] \to \R$ such that
\begin{align*}
    w_\eta (x) \eqdef \underset{s \neq t}{\sup} \, \frac{\abs{x(s) - x(t)}}{\abs{s-t}^\eta} < \infty
\end{align*}
and try to establish its weak convergence to a standard Brownian motion $W$ in this function space. 
However, it is worth noting that the Hölder spaces are not separable endowed with the usual norm defined by $N_\eta(x) \eqdef w_\eta(x) + \abs{x(0)}$ neither with any norm $N$  verifying $N(.) \geq cN_\eta(.)$ with $c>0$. To overcome this issue, we follow what has been done in \cite{Cie1960}: we define the analogue of the continuity modulus for Hölderian continuous functions by 
\begin{align*}
    w_\eta (x, \delta) 
    \eqdef 
    \underset{0 < \abs{s-t} < \delta}{\sup} \frac{\abs{x(s) - x(t)}}{\abs{s-t}^\eta}
\end{align*}
and then define $\mathcal H _\eta^0 [0,1] \eqdef \{x \in \mathcal H _\eta [0,1] \, ; \, \lim_{\delta \to 0} w_\eta (x, \delta) = 0\}$. This space endowed with the norm $\norm{x}_\eta \eqdef w_\eta (x, 1) + \abs{x(0)}$, is a separable Banach space. Moreover, the canonical embedding $\mathcal H _\eta^0 [0,1] \to \mathcal H _\eta [0,1]$ is continuous, so that convergence in distribution in $\mathcal H _\eta^0 [0,1]$ also takes place in $\mathcal H _\eta [0,1]$. 
The study of weak invariance principle in these functions spaces has been initiated for i.i.d. random variables by Lamperti (\cite{Lam1962}). Hamadouche (\cite{Ham2000}) extended the weak Hölderian invariance principle to dependent sequences under some sufficient strong mixing conditions. For recent results, see for instance Ra\v{c}kauskas and Suquet (\cite{RS2004}).
The case of $\tau$-mixing sequences has been studied by \parencite[]{Gir2017} whereas the case of $\alpha$-dependent sequences has been established in \parencite[]{DM2017}. In the context of dynamical systems and for bounded observables, the result of Giraudo can be written:
\begin{thm}[Giraudo \cite{Gir2017}]
        \label{thm_gir}
    Let $p > 2$. Assume that $(\varphi \circ f^n)_{n \geq 0}$ is a strictly stationary centered sequence such that 
        \begin{align}\label{cond_Gir_tau}
            \lim\limits_{n \to \infty} n^{p-1}\tau_\varphi(n) = 0. 
        \end{align}
    Then 
        \begin{align*}
            W_n(\varphi) \to \sigma(\varphi)W \text{ in distribution in } \mathcal H _{1/2-1/p}^0 [0,1], 
        \end{align*}
    where 
    $\sigma^2(\varphi) = Var(\varphi) + 2\sum_{k \geq 1} Cov(\varphi, \varphi \circ f^k)$ and $W$ is a standard Brownian motion. 
\end{thm}
\noindent
Consequently, taking into account Proposition \ref{prp:estim_tau_weakmoment} and if $s(n) \xrightarrow[n \to \infty]{} 0$, we get Theorem \ref{hip}.
\\

It is worth noting that compared to the definition taken by \cite{Gir2017} there is a time reversal in the $\tau$-mixing coefficients. Working with coefficients reversed in time is common when considering dynamical systems, but note that this time reversal is not binding for the application of Giraudo's theorem as its result deals with the distribution of the process. 

\comments{Recall the setting we consider: $f$ a $\mu$-invariant nonuniformly expanding transformation of $E$ (see Section 2 in \cite{CDM2023}), $\varphi$ a $\mu$-centered Hölder-continuous  observable. We consider $Y \subset E$ admitting an at most countable partition $\Gamma$ and a \textit{reference measure} $m$ such that $m(a ) > 0$ for any $a \in \Gamma$, and $R$ a return time to the basis $Y$ constant on each $a \in \Gamma$ and such that $m(R) < \infty$. }

With the precise computation of the $\tau$-mixing coefficients we can also prove the following Baum-Katz type result based on \parencite[Remark 2.1]{CDM2023} and \parencite[Theorem 2]{DM2006}. 
\begin{thm}\label{BaumKatz}
    Let $p > 1$. Assume that $\int R^p \, dm < \infty$. Let $\varphi$ be a centered Hölder continuous observable and $x > 0$, then 
    \begin{align*}
         \sum_{n \geq 1} n ^{a p -2} \mu \bigp{\sup_{1 \leq j \leq n} \abs{S_j(\varphi)} \geq xn^a} < \infty
    \end{align*}
    with 
     $a \in \left] \frac 1 2 , 1\right]$ if $p \geq 2$
     and 
     $a \in \left[ \frac 1 p , 1\right]$ if $1 < p < 2$. 
\end{thm}

By standard arguments using the Borel-Cantelli lemma and taking $a = 1/p$, we get the following corollary.

\begin{cor}
    Let $1 < p < 2$. Assume that $\int R^p \, dm < \infty$, then 
    \begin{align*}
        n^{-1/p}S_n(\varphi) \xrightarrow[n \to \infty]{} 0 \quad \mu-\text{almost surely}.
    \end{align*}
\end{cor}

\begin{proof}[Proof of Theorem \ref{BaumKatz}]
First, let us prove the following result on $\tau$-mixing coefficients mentioned in \parencite[Remark 2.1]{CDKM2023}.
\begin{lem}\label{remark2.1}
    Let $p > 1$. Assume that $\int R^p \, dm < \infty$, then 
    \begin{align*}
        \sum_{k=1}^{\infty}k^{p-2}\tau_{\varphi}(k) < \infty.
    \end{align*}
\end{lem}
\begin{proof}[Proof of the lemma]
    Taking $T$ as defined in (\ref{def_T}) and according to the upper bound (\ref{estim_tau_phi}) of Subsection \ref{subsubsec:estim_tau} with $r$ large enough, it is enough to prove that the series $\sum_{k\geq 1}k^{p-2}\Proba(T \geq k)$ converges.
    Applying Fubini, it is sufficient to show that $\E(T^{p-1}) < \infty$. 
    The result follows from \parencite[Lemma 3.1]{CDKM2020b}.
\end{proof}

Since our observables are bounded, Theorem \ref{BaumKatz} follows by combining \parencite[Theorem 2]{DM2006} with Lemma \ref{remark2.1}. 
\end{proof}
    
        \section{Proof of Proposition \ref{prp:estim_tau_weakmoment}}

\subsection{Reduction to Markov trajectories }

As done in different papers, we shall use the construction of a Markov chain in order to bring us back to the simpler framework of the study of partial sums $\sum X_k$. 

We refer to \cite{Kor2018} and \parencite[Part 2.3]{CDKM2020b} for the construction of this Markov chain through a Young towers' process. 
Let $\mathcal A$ denote the set of all finite words in $\Gamma$, not including the empty word. For $\omega = a_0 \cdots a_{n-1} \in \mathcal A$ we denote $h(\omega) = R(a_0) + \cdots + R(a_{n-1})$. 
There exists a probability measure $\Proba_{\mathcal A}$ that weighs any element of $\mathcal A$ and there exists $\xi \in ]0,1[$ such that $\Proba_{\mathcal A}(h = n) = (1- \xi)^n \xi$ for any $n \geq 0$. Furthermore, note that according to Proposition 4.4 and Proposition 4.5 in \cite{Kor2018}, the return time tails are closely related to those of $h$:
\begin{itemize}
    \item if $m(R \geq n) = O(n^{-p})$ with $p > 1$, then $ \Proba_{\mathcal A}(h \geq n) = O(n^{-p})$, 
    \item if $\int R^p dm < \infty$ with $p > 1$, then $\int h^p d\Proba_{\mathcal A} < \infty$. 
\end{itemize}
We will adapt the proof to our context and show in Proposition \ref{queue_h}  that if, for any $n$, one has $e^{-nu_n} \leq m(R > n) \leq n^{-p}s(n)$ with $p > 0$, $s$ a slowly varying function and $(u_n)_n$ a sequence of positive reals that converges to 0, then for any $n$, $
        c_1 e^{-nu_n} \leq \Proba_{\mathcal A}(h \geq n) \leq c_2 n^{-p}s(n)
    \label{link_h_R_extended}$
for some $c_1, c_2 >0$.

As done in \cite{CDKM2020b} for the particular case of nonuniformly expanding dynamical systems, the random process $S_{n}(\varphi)$ can be redefined on a Markov shift without changing its distribution, whose structure is as follows.
Consider $S=\{(w, \ell) \in \mathcal{A} \times \mathbb{Z}: 0 \leq \ell<h(w)\} $ and define the separation metric $d_{\mathcal{A}^{\mathbb{N}}}$ by $d_{\mathcal{A}^{\mathbb{N}}}((w_n)_n,(w_n')_n) = \xi ^{-\inf \{ n : w_n \neq w_n' \}}$. According to \parencite[Corollary 2.5]{CDKM2020b}, given $f$ a nonuniformly expanding map and $\varphi$ a Hölder continuous observable, there exist a stationary Markov chain $(g_n)_{n \geq 0}$ on $S$ 
with stationary measure $\nu$
and $\psi$ a real-valued Hölder continuous function with respect to  $d_{\mathcal{A}^{\mathbb{N}}}$ defined on the space $\Omega$ of possible trajectories of $(g_n)_n$ such that 
    \begin{align}\label{equal_distrib}
        \left\{\varphi \circ f^{n}\right\}_{n \geq 0} \stackrel{\law}{=}\left\{\psi\left(g_{n}, g_{n+1}, \ldots\right)\right\}_{n \geq 0}
    \end{align}
respectively on $([0,1], \mu)$ and $(\Omega, \Proba_\Omega)$. 
Moreover, we can assume that $\left(g_{n}\right)_{n \geq 0}$ is generated by a sequence of independent innovations as follows. Let $g_{0} \sim \nu$ and let $\varepsilon_{1}, \varepsilon_{2}, \ldots$ be a sequence of independent (also from $g_{0}$) and identically distributed random variables with values in $\mathcal{A}$ and distribution $\mathbb{P}_{\mathcal{A}}$. Consider $U:(S, \mathcal A) \to S$ defined by 
\begin{equation*}
    U((\omega, \ell), \varepsilon) = \left\{
    \begin{array}{ll}
        (\omega, \ell +1) & \text{if } h(\omega) < \ell -1 \\
        (\eps, 0) &  \text{if } h(\omega) = \ell -1
    \end{array}
    \right. 
\end{equation*}
and for any $n\in \N$, $g_{n+1}=U\left(g_{n}, \varepsilon_{n+1}\right)$. Then $(g_n)_n$ has the same distribution as the previous Markov chain. 
Finally, note that $\left(g_{n}\right)_{n \geq 0}$ is $\beta$-mixing in the sense
\begin{align*}
    \beta(n) = \frac 1 2 \nu \bigp{\underset{\norm{\phi}_\infty \leq 1}{\sup}\abs{K^n(\phi) - \nu(\phi)}} \xrightarrow[n \to +\infty]{} 0
\end{align*}
where $K$ is the Kernel associated to the Markov chain. 
\noindent
If we denote
$$
X_{n} \eqdef \psi\left(g_{n}, g_{n+1}, \ldots\right) \quad \text { and } \quad S_{n}=\sum_{k=1}^{n} X_{k},
$$
the Hölderian invariance principle for $S_{n}(\varphi)$ is now reduced to the one for $S_{n}$. Note that the Hölder property of $\varphi$ is fundamental to get (\ref{equal_distrib}).

 \subsection{Computation of the \texorpdfstring{$\tau$}{t}-mixing coefficients}
    
    To ensure that (\ref{estim_tau_weakmoment}) holds, we will work on intermediate inequalities linking respectively the variables $R$, $h$, the meeting time (as defined above) and the $\tau_\varphi$-mixing coefficient. 
    \\
In what follows, we consider the following definitions: 
\begin{align}
        \label{def_tribuFG}
    \calF_i = \sigma( X_k \, : \, k \leq i)
    \qquad \text{and} \qquad 
    \mathcal G_i = \sigma(X_k \, : \, k \geq i). 
\end{align}
With these notations and with \refequ{equal_distrib} in mind, we can write 
\begin{align*}
    \tau_\varphi(n) = \underset{l \geq 1}{\sup} \frac{1}{l} \underset{1 \leq i_1 < \cdots < i_l}{\sup} \tau(\mathcal G _{i_l + n}, (X_{i_1}, \cdots, X_{i_l})). 
\end{align*}
    \\

    \noindent
    As a starting point, recall that for some $p >1$, there exist a slowly varying function $s$ that converges to 0 at infinity and a sequence $(u_n)_n$ that converges towards 0 such that
    \begin{align} \label{queue_m}
        m(R > n) = O(n^{-p} s(n)) \text{ and } m(R > n)  \geq e^{-nu_n}. 
    \end{align}
\subsubsection{Tail estimates of \texorpdfstring{$h$}{}}
Using Korepanov's arguments in \parencite[Section 4.3]{Kor2018} for strong and weak polynomial moments of $h$, we prove the following behavior of $h$ in our context.   
\begin{prp}\label{queue_h}
    There exist $c_1, c_2 >0$ such that for any $n$,
    \begin{align}\label{ineq_queue_h}
        c_1 e^{-nu_n} \leq \Proba_{\mathcal{A}}(h > n) \leq c_2 n^{-p}s(n)
        .
    \end{align}
\end{prp}
\begin{proof}
     From Lemma 4.3 in \parencite{Kor2018} 
    \begin{align*}
        \Proba_{\mathcal{A}}(h > n)  
        \leq c\xi\sum_{j=1}^{\infty} j \, m(R>n/j)(1-\xi)^j
        .
    \end{align*}
    We split the sum into two parts to leverage the asymptotic behavior of the slowly varying function $s$.
    On the one hand, taking into account \refequ{queue_m}, 
    \begin{align*}
        \sum_{j \leq n^{1/2}} j \, m(R>n/j)(1-\xi)^j 
        &\leq Cn^{-p} \sum_{j \leq n^{1/2}} j^{1+p}(1-\xi)^j s(n/j).
    \end{align*}
    As $s$ is slowly varying, and both $n,n/j \geq n^{1/2}$ can be arbitrary large, for $n$ large enough, by applying Lemma \ref{Potter_thm} we get
    $\frac{s(n/j)}{s(n)} \leq 2j$.
    Thus, 
    \begin{align*}
        \sum_{j \leq n^{1/2}} j^{1+p}(1-\xi)^j \frac{s(n/j)}{s(n)} 
        \leq 
        C\sum_{j \geq 1} j^{2+p} (1-\xi)^j
        < \infty. 
    \end{align*}
    Note that, $ j(1-\xi)^j = o((1-\xi/2)^j) $ so that $\sum_{j > n^{1/2}} j(1-\xi)^j = o(n^{-p-1})$.
    Hence we obtain the following upper bound:
    \begin{align*}
        \Proba_{\mathcal{A}}(h > n) \leq Cn^{-p}s(n). 
    \end{align*}
   On the other hand, one has
    $
       \Proba_{\mathcal{A}}(h > n) \geq \Proba_{\mathcal{A}}(\mathcal A_0, R(a_0) > n)
    $
    where $\mathcal A_0$ denotes the set of words of length 1.
    Hence, applying \parencite[Lemma 4.3]{Kor2018},
   $
        \Proba_{\mathcal{A}}(h > n) \geq \tilde c\nu_Y(R > n).
    $
    Finally, combining the latter inequality with the estimate $c^{-1} \leq d\nu_Y/dm \leq c$ for some $c > 0$ and  (\ref{queue_m}), we get $ \Proba_{\mathcal{A}}(h > n) \geq \Tilde Ce^{-nu_n}$ for some $\Tilde C > 0$.
\end{proof}

\subsubsection{Tail estimates of the meeting time}

\noindent 
Consider $g_0^*$ a random variable with distribution $\nu$, independent of $(g_0, (\eps_n)_{n \geq 1})$ and define the Markov chain $(g_n^*)_n$ by $g_{n+1}^* = U(g_n^*, \eps_{n+1})$. Define the meeting time as in \parencite[]{CDKM2023}:
\begin{align}\label{def_T}
    T = \inf \{ n \geq 0 \, : \, g_n = g_n^*\}. 
\end{align}
We aim to estimate $\Proba(T \geq n)$. 
Using both lower and upper bounds in (\ref{ineq_queue_h}), we can apply \parencite[Lemma 3.6]{CDKM2023}. We get
\begin{align}
        \label{queue_T-h}
    \Proba(T \geq n) = O\bigp{\E_{\mathcal A}\bigcro{(h - n)_+}}. 
\end{align}
\\
Now, for any $n$,  
$
    \E_{\mathcal A}\bigcro{(h - n)_+} 
    = \sum_{k > n} \Proba_{\mathcal A}(h \geq k)
$ 
and combining Lemma \ref{Potter_thm}
with Proposition \ref{queue_h}, for $n$ large enough,
\begin{align}
    &\hspace{15pt}\sum_{k>n}\Proba_{\mathcal{A}}(h \geq k ) 
    \leq c_2\sum_{k>n}k^{-p}s(k)
    \leq 2^{(p-1)/2}c_2 \frac{s(n)}{n^{(p-1)/2}}\sum_{k>n} k^{-(p+1)/2} \label{major_series_h}
\end{align}
where the remainder can be explicitly calculated. 
\noindent 
Hence, from (\ref{queue_T-h}) and (\ref{major_series_h}),
\begin{align}\label{controle_queue_T}
    \Proba(T \geq n) = O(n^{1-p}s(n)). 
\end{align}

Finally, we establish a connection between the $\tau_\varphi$-mixing coefficients and the tail behavior of $T$, which allows us to prove that (\ref{estim_tau_weakmoment}) holds. 

\subsubsection{Estimation of the \texorpdfstring{$\tau_\varphi$}{}-mixing coefficients}
    \label{subsubsec:estim_tau}

\noindent
Let $\psi$ be the Hölder continuous observable defined in \refequ{equal_distrib}.  
Let $(\eps_n')_n$ be an independent copy of $(\eps_n)_n$, independent of $g_0$. 
In what follows, we denote  by $(g_i, ..., g_{i+k}, g'_{i+k+1}, g'_{i+k+2}, ...)$ the Markov chain where we use the innovation $(\varepsilon'_n)_n$ from time $i+k+1$ to generate the trajectories of the chain (with this notation $g'_{i+k+1}=U(g_{i+k}, \varepsilon'_{i+k+1})$, $g'_{i+k+2}= U(g'_{i+k+1}, \varepsilon'_{i+k+2})$, and so on...).
For $n \geq 0$, define 
\begin{align*}
    &\delta(n) \eqdef \E\left( \abs{
        \psi(g_0, g_1, \cdots, g_n, g_{n+1}, g_{n+2}, \cdots) - \psi(g_0, g_1, \cdots, g_n, g_{n+1}', g_{n+2}', \cdots)
    }\right)
\end{align*}
that compares Markov chains with the same $n$ first terms but different innovations from rank $n+1$. 
In what follows, another useful tool will be $\beta$-mixing coefficients defined as follows:
\begin{align*}
    \beta(n) = \underset{k \in \Z}{\sup} \,  \beta(\calF_k, \mathcal G_{n + k}) = \frac 1 2 \nu \bigp{\underset{\norm{\phi}_\infty \leq 1}{\sup}\abs{K^n(\phi)}},
\end{align*}
where 
    $\beta({\mathcal U}, {\mathcal V})= 1/2 \, \E(\sup\{ |\E(V|{\mathcal U})-\E(V)| \, : \, V \in L^{\infty}({\mathcal V}), ||V||_{\infty} \leq 1\})
    $
and $(\calF_n)_n$, $(\mathcal G _n)_n$ are defined in (\ref{def_tribuFG}).
\\
First, let us establish an estimation of $\delta(n)$ as has been done when $\E[T] < \infty$ in \parencite[Proposition 3.2]{CDKM2020b}.
\begin{lem}\label{control_delta}
Assume that $T$ admits a finite (strong) moment of order $q > 0$. Then, for any $r \geq 1$, 
        \begin{align*}
            \delta(n) \leq C(r)n^{-\frac{r\min(q,1)}{2}} + C(r)\Proba(T \geq [n/r])
        \end{align*}
    where $C(r)$ does not depend on $n$.
\end{lem}

The case $q \geq 1$ has already been established in \parencite[Proposition 3.2]{CDKM2020b} based on a Fuk-Nagaev type theorem due to \cite{Rio2017}. The same proof with slightly more precise estimates will extend the result to the case $ 0 < q < 1$, let us give the details. 
\begin{proof}[Proof of Lemma 3.2 (0<q<1)]
According to (3.5) in \cite{CDKM2020b}, there exist $c > 0$ and $a \in ]0,1[$ such that
\begin{align*}
    \delta(n) \leq c a^n + c \Proba \left( \abs{ \sum_{i=0}^{n} \indic_{g_i \in S_0} - (n+1)\nu(S_0)} > \frac{1}{2}(n+1)\nu(S_0) \right) 
\end{align*}
where $S_0 = \mathcal{A} \times \{0\}$. 
We apply Fuk-Nagaev's Theorem 6.2 from \parencite{Rio2017} to the bounded random variables $\indic_{g_i \in S_0}$. Denoting 
    $s_n^2 = \sum_{0\leq i,j \leq n} \abs{\text{Cov}(\indic_{g_i \in S_0}, \indic_{g_j \in S_0})}$, 
we obtain for any $r \geq 1$,
\begin{align*}
    \Proba \bigp{ \abs{ \sum_{i=0}^{n} \indic_{g_i \in S_0} - (n+1)\nu(S_0)} > \frac{1}{2}(n+1)\nu(S_0) } \leq C\bigp{\frac{s_n^2}{n^2}}^{r/2} + C\Proba(T \geq [n/r]).
\end{align*}
Now, the covariance terms $|\text{Cov}(\indic_{g_i \in S_0}, \indic_{g_j \in S_0})|$ are bounded by $2\beta(|i-j|)$, thus we infer that $s_n^2 \leq 4(n+1)\sum_{i=0}^{n}\beta(i) $. 
But through estimate (3.6) in \parencite[]{CDKM2020b}, we have 
\begin{align}
        \label{control_beta}
    \beta(i) \leq \Proba(T\geq i) \leq \E[T^{q}]i^{-q}.
\end{align}
Since $q \in ]0,1[$, it follows that $s_n^2 = O(n^{2-q})$. 
We finally obtain for any $r \geq 1$
\begin{align*}
    \delta(n) \leq C(r) \bigcro{n^{-\frac{rq}{2}} + \Proba(T \geq [n/r])}.
\end{align*}
\end{proof}
~\\

We now have the tools to estimate the mixing coefficients $\tau_{\varphi}(n)$ following the proof of Proposition 2.1 in \parencite[]{CDKM2023}{}.
The control of $\delta(n)$ established in Lemma \ref{control_delta} allows us to replicate the proof of \parencite[Proposition 2.1]{CDM2023} with more details. Since $\Proba(T \geq n) = O(n^{1-p}s(n))$, $T$ has a strong moment of order  $(p-1)/2$. We can then apply Lemma \ref{control_delta} with $q = (p-1)/2 > 0$. Taking $r$ large enough, $n^{-r\min(q,1)/2}$ can be arbitrarily small so that with \refequ{controle_queue_T} in mind, it is enough to prove that: 
    \begin{align}
            \label{control_tau_deltabeta}
        \tau_\varphi(n) = O(\delta(\lfloor n/2 \rfloor) + \beta(\lfloor n/2 \rfloor)). 
    \end{align}
    Let $X_{j,n}' = \psi\left(g_j, g_{j+1}, \cdots, g_{j+\lfloor n/2 \rfloor}, g_{j+\lfloor n/2 \rfloor +1}', g_{j+\lfloor n/2 \rfloor +2}', \cdots\right)$, according to (6.3) in \cite{CDKM2023}, we have 
    \begin{align}
            \label{control_tau_decomposition}
        \tau_{\varphi}(n) \leq 2\, \delta(\lfloor n/2 \rfloor) + \sup_{l \geq 1}\frac{1}{l}\sup_{1\leq i_1<\cdots<i_l} \tau\bigp{\mathcal G_{i_l + n}, (X_{i_1,n}', \cdots, X_{i_l,n}')}. 
    \end{align}
    Let us give some details. 
    From stationarity, note that for any $j$, 
    \begin{align*}
        \E \abs{X_{j,n}' - X_{j}} \leq \delta(\lfloor n/2 \rfloor).
    \end{align*}
    Consequently, from the definition of $\tau$, 
    \begin{align*}
        \abs{
            \tau(\mathcal G _{i_l+n},(X_{i_1}, \cdots, X_{i_l})) - \tau(\mathcal G _{i_l+n},(X_{i_1,n}', \cdots, X_{i_l,n}')) 
        }
        \leq 
        2l\, \delta(\lfloor n/2 \rfloor)
    \end{align*}
    and (\ref{control_tau_decomposition}) follows. 
    Now, by Berbee's coupling lemma, there exists $(X_{i_1}^*, \cdots, X_{i_l}^*)$ distributed as $(X_{i_1,n}', \cdots, X_{i_l,n}')$, independent of $\mathcal G _{i_l+n}$ and such that 
    \begin{align*}
        2\Proba((X_{i_1,n}', \cdots, X_{i_l,n}') \neq (X_{i_1}^*, \cdots, X_{i_l}^*)) 
        &= \beta(\mathcal G _{i_l+n}, \sigma(X_{i_1,n}', \cdots, X_{i_l,n}'))
        \\&
        \leq \beta(\mathcal G _{i_l+n}, \mathcal F_{i_l + \lfloor n/2 \rfloor}')
        \leq \sup_{k \geq 0} \beta(\calF _k', \mathcal G _{k + \lfloor n/2 \rfloor})
    \end{align*}
    with $\calF_i' \eqdef \calF_i \vee \sigma\bigp{(\eps_i')_{i \geq 1}}$. 
    \\
    Note that by the Kantorovich-Rubinstein duality theorem,
    \begin{align*}
        \tau\bigp{\mathcal G_{i_l + n}, (X_{i_1,n}', \cdots, X_{i_l,n}')} 
        &= 
        W_1 \bigp{P_{(X_{i_1,n}', \cdots, X_{i_l,n}')|\mathcal G_{i_l + n}}, P_{(X_{i_1,n}', \cdots, X_{i_l,n}')}},  
    \end{align*}
    where $W_1$ is the Kantorovitch-Rubinstein distance defined for any $\mu_1, \mu_2$ real probability measures by $W_1(\mu_1, \mu_2) = \inf\{\E \, |Y_1-Y_2| : \law(Y_1) = \mu_1, \law(Y_2) = \mu_2\}$.
    Then
    \begin{align*}
        \tau\bigp{\mathcal G_{i_l + n}, (X_{i_1,n}', \cdots, X_{i_l,n}')} 
        &\leq \E\bigp{\norm{(X_{i_1,n}', \cdots, X_{i_l,n}')-(X_{i_1}^*, \cdots, X_{i_l}^*)}_{\ell^1}}
        \\
        &\leq \E \bigcro{
            \bigp{\sum_{k = 1}^l \abs{X_{i_k,n}' - X_{i_k}^*}} \indic_{(X_{i_1,n}', \cdots, X_{i_l,n}') \neq (X_{i_1}^*, \cdots, X_{i_l}^*)}
        }
        \\
        &\leq 2l \, \norm{\psi}_\infty \Proba((X_{i_1,n}', \cdots, X_{i_l,n}') \neq (X_{i_1}^*, \cdots, X_{i_l}^*)) 
        \\
        &\leq 
        l \, \norm{\psi}_{\infty}\, \sup_{k \geq 0} \beta(\calF _k', \mathcal G _{k + \lfloor n/2 \rfloor}). 
    \end{align*}
    Since $(\eps_i')_i$ is independent of $(g_i)_i$, $\beta(\calF _k', \mathcal G _{k + \lfloor n/2 \rfloor}) = \beta(\calF _k, \mathcal G _{k + \lfloor n/2 \rfloor})$ (see \parencite[Theorem 6.2]{Bra2007}). Taking the supremum over $k$, we infer that 
    \begin{align}
            \label{control_tau_annex}
        \tau\bigp{\mathcal G_{i_l + n}, (X_{i_1,n}', \cdots, X_{i_l,n}')} \leq l \norm{\psi}_\infty  \beta(\lfloor n/2 \rfloor). 
    \end{align}
    Finally, combining (\ref{control_tau_decomposition}) and (\ref{control_tau_annex}), we infer that (\ref{control_tau_deltabeta}) is verified. Now, (\ref{control_tau_deltabeta}) together with (\ref{control_beta}) and Lemma \ref{control_delta} with $r$ large enough imply 
    \begin{align}\label{estim_tau_phi} 
        \tau_{\varphi}(n) = O(n^{-r/2}) + \Proba(T \geq \lfloor n / r \rfloor) + \Proba(T \geq \lfloor n / 2 \rfloor)). 
    \end{align}
    Taking into account \refequ{controle_queue_T} and \refequ{estim_tau_phi} with $r > (p-1)/2$, 
    \begin{align*}
        \tau_{\varphi}(n) = O(n^{1-p}s(n)). 
    \end{align*}
    This completes the proof of Proposition \ref{prp:estim_tau_weakmoment}. 

\section{Hölderian invariance principle for BV observables}
    \label{Section:BV}

Maps in $Hol(\gamma, \rho)$ are a generalization of LSV maps introduced by Liverani, Saussol and Vaienti in \cite{LSV} (for the LSV map, $\rho =2^\gamma$). First studied by Holland in \cite{Hol05}, they were then particularly studied by Gouëzel in his thesis \cite{GouThese} (whose results will be used extensively).
In the following, we will restrict ourselves to the case $\gamma < 1$ in which Gouëzel proved that there exists a unique $f$-invariant probability $\nu$ which is absolutely continuous with respect to the Lebesgue measure $\lambda$, with a positive and Lipschitz continuous density $h$ on any subinterval $[\varepsilon, 1]$ (see \parencite[Theorem 1.4.10]{GouThese}). The return time at $Y=]1/2,1]$ defined by $R(x) = \inf \{n \geq 1 : f^n(x) \in Y \}$ is finite almost everywhere on $Y$ and the induced map $F_Y = f^{R}$ is then a uniformly expanding map of $Y$ and admits an invariant probability  $\nu_0$ equivalent to $\lambda$ on $Y$. It is worth noting that $\nu$ can be constructed by renormalizing the following 
\begin{align*}
    \nu_0(A\cap Y) + \sum_{n \geq 1} \nu_0 (f^{-n}(A) \cap \{ R > n \}) 
\end{align*}
(see for instance \parencite[Lemma 9]{Zwe1998}). 
Let us consider $v_0 : ]0,1[ \to ]0,1/2[$ and $v_1 : ]0,1[ \to ]1/2,1[$ the inverse branches of $f$. Then by the above construction of $\nu$, for any $x \in [0,1/2]$, 
\begin{align}\label{eq_densite_h}
    h(x) = \sum_{n=0}^{\infty} \abs{(v_1 \circ v_0^n)'(x)}h(v_1 \circ v_0^nx).
\end{align}
Let us define inductively a sequence $(z_n)_n$ by $z_0 = 1$, $z_1= 1/2$ and $z_{n+1} = v_0(z_n)$. Then, $f^n$ is a bijection between $J_n \eqdef  ]z_{n+1}, z_n]$ and $Y$. In \cite{GouThese}, Gouëzel showed first in Proposition 1.4.7 that $z_n^{\gamma} \rho(z_n) \sim \frac{1}{\gamma . n}$, which combined with Theorem 1.4.10, as mentioned in Section \ref{Section:intro}, implies that
\begin{align}\label{queue_tps_retour}
    \nu(R > n) = n^{-p}s(n),
\end{align}
where $p = 1/\gamma$ and $s$ is a slowly varying function at $+\infty$, only depending on $\rho$ and $\gamma$ and linked to the de Bruijn conjugate of $\rho$. Its asymptotic behavior only depends of the one of $\rho$. 
Note that as soon as $\rho$ converges to 0 at 0, $s$ converges to 0 at infinity according to \parencite[Section 1.4.2]{GouThese}.
Moreover, we shall use the following facts proved in \cite{LSV}:
\begin{align}
        \label{estim_zn}
    (z_n - z_{n+1}) \sim Cn^{-p-1}s(n)
    \qquad\text{and}\qquad
    h \sim Cn \text{ on } [z_{n+1},z_n]. 
\end{align}
Following what has been done in \cite{DGM2010}, we will estimate some weak mixing coefficients in order to quantify the mixing properties of the iterates of $f$. We will study the action of the transfer operator associated to our dynamical system on the space $\mathcal{BV}$ of functions with bounded variation on $[0,1]$ based on results of \cite{Gou2007}. 
For $\varphi \in L^1([0,1])$  and $x \in [0,1]$, the transfer operator $K$ is defined by 
\begin{align*}
  K\varphi(x) &= \frac{1}{h(x)}\sum_{f(y)=x}\frac{h(y)}{\abs{f'(y)}}\varphi(y) \\
  &= \frac{1}{h(x)}\Big( h(v_0(x))\abs{v_0'(x)}f(v_0(x)) + h(v_1x)\abs{v_1'(x)}\varphi(v_1(x)) \Big). 
\end{align*}
\newline
The operator $K$ acts continuously on $L^1([0,1])$, and we will see that it also acts on $\mathcal{BV}$.
For the reader's convenience, let us recall some results on functions with bounded variation. If $\varphi$ is supported on $[0,1]$ we define its variation on $\R$ by
$$
    V(\varphi) = \underset{x_0 < \ldots < x_N}{\sup} \sum_{k=0}^{N-1}\abs{\varphi(x_{k+1}) - \varphi(x_k)}
$$
where the $x_k$’s are real numbers (not necessarily in $[0, 1]$). Then, by denoting $\norm{\cdot}_\infty$ the supremum norm on $[0,1]$:
\begin{itemize}
    \item $V(\varphi) \geq 2\lVert \varphi \rVert_{\infty}$,
    \item $V(\varphi \psi) \leq \norm{\varphi}_{\infty} V(\psi) + V(\varphi) \norm{\psi}_{\infty}\leq V(\varphi)V(\psi)$,
    \item  $V(\varphi^{(0)}) \leq 3\norm{d\varphi}$, where $\norm{d\varphi}$ is the variation norm of the signed measure $d\varphi$ associated to $\varphi$ on $[0,1]$ and $\varphi^{(0)} = \varphi - \nu(\varphi)$. 
\end{itemize} 
In order to prove the last item above write $\varphi = \psi_1 -\psi_2$ with $\psi_1, \psi_2$ bounded and non-decreasing. Then $V(\varphi^{(0)}) \leq V(\psi_1^{(0)}) + V(\psi_2^{(0)})$ and $\norm{d\varphi} = \norm{d\psi_1} + \norm{d\psi_2}$. Thus it is sufficient to show the inequality for $\psi_1$:  $V(\psi_1^{(0)}) = (\psi_1(1) - \psi_1(0)) + (\psi_1(1)-\nu(\psi_1)) + (\nu(\psi_1) - \psi_1(0)) \leq 3\norm{d\psi_1}$.
\\

In the following, we denote $\mathcal{BV}_1$ the space of functions $\varphi \in \mathcal{BV}$ such that $\norm{d\varphi} \leq 1$, and for $\varphi_1,\ldots,\varphi_j$ bounded measurable functions and $n_1,\ldots,n_j$ nonnegative integers, we set
\begin{align*}
    K^{(0)(n_1,\ldots,n_j)}(\varphi_1,\ldots,\varphi_j) = (K^{n_1}(\varphi_1K^{n_2}(\varphi_2K^{n_3}(\varphi_3\ldots \varphi_{j-1}K^{n_j}(\varphi_j)))))^{(0)},
\end{align*}
where $\varphi^{(0)} = \varphi - \nu(\varphi)$. 
As in \cite{DGM2010} (see their formula (1.15)), we define the $\alpha$-dependent coefficients associated with the Markov chain $(Y_k)_{k \geq 0}$, with stationary measure $\nu$ and transfer operator $K$ as follows
   \begin{align}
       \alpha_{k, \mathbf{Y}}(n) = \underset{1 \leq j \leq k}{\sup} \underset{\substack{n_1 \geq n \\ n_2,\ldots,n_j \geq 0}}{\sup} ~\underset{\varphi_1,\ldots,\varphi_j \in \mathcal{BV}_1}{\sup} \nu \Big( \abs{ K^{(0)(n_1,\ldots,n_j)}(\varphi_1^{(0)},\ldots,\varphi_j^{(0)})} \Big). 
   \end{align}
   Note that these coefficients can be rewritten 
    \begin{align*}
        \alpha_{k,\mathbf{Y}}(n) \eqdef \underset{1\leq l \leq k}{\max} \underset{n \leq i_1 \leq \ldots\leq i_l}{\sup} \alpha(\mathcal{F}_0,(Y_{i_1},\ldots,Y_{i_l}))
    \end{align*}   
    where for any random variables $Z=(Z_1,\ldots,Z_k)$ with values in $\mathbb{R}^k$ and any $\sigma$-algebra $\mathcal{F}$,
    \begin{align*}
        \alpha(\mathcal{F},Z) \eqdef \underset{(x_1,\ldots x_k) \in \mathbb{R}^k}{\sup} \Bigg \lVert \condEp{ \prod_{j=1}^k (\indic_{Z_j \leq x_j})^{(0)}}{\calF}^{(0)} \Bigg \rVert_1
    \end{align*}    
    and $U^{(0)} = U - \E(U)$. 
    
Let us follow \parencite[Section 3]{DGM2010} to estimate $\alpha_{k, \mathbf Y}(n)$ in order to establish the following result:
\begin{thm}\label{thm_alpha}
    Let $f$ be a Holland map with return time $R$ such that $\nu(R > n) \sim n^{-p}s(n)$, with $p > 2$ and $s$ slowly varying at infinity. 
    Then for any $k \geq 1$, there exists $C(k) > 0 $ such that
    \begin{align*}
        \alpha_{k, \mathbf Y}(n) \leq \frac{C(k)s(n)}{n^{p-1}}.
    \end{align*}
\end{thm}
With this last estimate, it is sufficient to apply Theorem 1.13 in \cite{DGM2010} to ensure the absolute convergence of the series defined in Theorem \ref{hip_BV}, as well as an invariance principle in the space of continuous functions. Furthermore, Theorem 3.3 in \cite{DM2017} refines the convergence into the space $\mathcal H_{1/2-1/p}^0([0,1])$. This proves our Theorem \ref{hip_BV}.\\

Let us make a study of the iterates $K^n$ of the transfer operator.
Recall that $v_0^n$ is an increasing function. 
The functions $\abs{v_1'}$ and $h$ are bounded above and below on $Y$, hence by (\ref{eq_densite_h}) there exists $C \geq 1$ such that for any $x \in [0,1/2]$
\begin{align}\label{encadrement_h}
    \frac{1}{C} \sum_{n=0}^{\infty} {(v_0^n)'(x)} \leq h(x) \leq C \sum_{n=0}^{\infty} {(v_0^n)'(x)}.
\end{align}
Note that in this section $C$ may denote different constants.
\\
Let us introduce the kernel associated to $K$ that we will abusively denote $K$ as well: 
\begin{center}
    $K(x,v_ix) = \frac{h(v_ix)\abs{v_i'(x)}}{h(x)}$, $i=0,1$, and $K(x,y) = 0$ if $y$ is not an antecedent of $x$ for $f$.
\end{center}
By definition, $K^n$ satisfies for any $x \in [0,1]$:
\begin{align}\label{Kn_def_sum}
    K^n \varphi(x) = \sum_{x_1,\ldots,x_n} K(x,x_1)\cdot \cdot \cdot K(x_{n-1},x_n)\varphi(x_n).
\end{align}
In order to study the operators $K^n$, we cut the previous sum into different "trajectories" $x,x_1,\ldots,x_n$, reproducing the method of \parencite[Section 3]{DGM2010}. Let us define the following operators as $K^n$ is in (\ref{Kn_def_sum}) with:
\begin{itemize}
    \item for $A_n$: $x \in [0,1/2]$, $x_1\ldots,x_{n-1} \in [0,1/2]$ and $x_n \in Y$;
    \item for $B_n$: $x \in Y$ and $x_1,\ldots,x_n \in [0,1/2]$;
    \item for $C_n$: $x \in [0,1/2]$ and $x_1,\ldots,x_{n} \in [0,1/2]$;
    \item for $T_n$: $x \in Y$, $x_1,\ldots,x_{n-1} \in [0,1]$ and $x_n \in Y$.
\end{itemize}
We can now decompose $K^n$ as follows:
\begin{align}\label{itérées_Kn}
    K^n = \sum_{a+k+b =n} A_a \circ T_k \circ B_b + C_n.
\end{align} 
Consequently, the study of $K^n$ will follow from the studies of each of the four operators. We will differentiate the study of $T_n$ from that of $A_n$, $B_n$ and $C_n$ which have the advantage of being written explicitly from $h$, $v_0$ and $v_1$.

\subsection{Study of \texorpdfstring{$T_n$}{}}
    \label{section:decomp_Tn}

In order to obtain a description of the operators $T_n$, we will show that they verify a renewal equation that allows us to apply an operator decomposition result (\parencite[Theorem 1.1]{Gou2004}) following Gouëzel's works in \cite{Gou2004} and \cite{Gou2007}. 

First, we can state a preliminary decomposition analogous to \parencite[Proposition 3.1]{DGM2010}. 
\begin{prp}\label{decomp_Tn}
    The operator $T_n$, acting on $\mathcal{BV}$, can be decomposed in the following way:
    \begin{align*}
        T_n = Q + E_n 
    \end{align*}
    where  $Q\varphi = \Big( \int_{Y} \varphi d\nu \Big) \indic_{Y}$ and $ \lVert E_n \rVert_{\mathcal L(\mathcal{BV})}\leq Cn^{-(p-1)}s(n)$.
\end{prp}
In order to prove Proposition \ref{decomp_Tn}, we aim to apply \parencite[Theorem 1.1]{Gou2004} that can be stated as below in a simplified version, sufficient for our purposes. 
\begin{thm}[Gouëzel, 2004]\label{thm_renewal}
    Let $(T_n)_n$ be bounded operators of a Banach space $\mathbb{B}$ such that for any $z \in \mathbb{D}$, $T(z) = id + \sum_{n \geq 1} z^n T_n$ converges in $\text{Hom}(\mathbb{B},\mathbb{B})$. Suppose that the following conditions are satisfied:
    \begin{itemize}
        \item \emph{Renewal equation}: for any $z \in \mathbb{D}$, $T(z) = (id-\Lambda(z))^{-1}$, where $\Lambda(z) = \sum_{n\geq 1} z^n \Lambda_n$, $\Lambda_n \in \text{Hom}(\mathbb{B},\mathbb{B}) $ and $\sum_{k>n} \lVert \Lambda_k \rVert_{\mathcal{L}(\mathbb{B})} = O(n^{-q})$, where $q > 2$.
        \item \emph{Spectral gap}: $1$ is a simple isolated eigenvalue of $\Lambda(1)$.
        \item \emph{Aperiodicity}: for any $z \in \Bar{\mathbb{D}}\backslash \{1 \}$, $id-\Lambda(z)$ is invertible.
    \end{itemize}
    Let $P$ be the projection of $\Lambda(1)$ associated to the eigenvalue $1$. If $P\Lambda'(1)P\neq 0$, considering $\mu$ such that $P\Lambda'(1)P = \mu P$, one has for any $n$
    \begin{align*}
        T_n = \frac{1}{\mu}P + \frac{1}{\mu^2}\sum_{j = n+1}^\infty(j-(n+1))P\circ \Lambda_j \circ P + N  \text{, where } \lVert N \rVert_{\mathcal L(\mathbb{B})} = O(n^{-q}). 
    \end{align*}
\end{thm}
Let us verify in our context that each hypothesis of Theorem \ref{thm_renewal} is fulfilled for $\mathbb{B} = \mathcal{BV}$ and $T = f$. 
Observe that $T_n\varphi = \indic_{Y} K^n(\indic_{Y}\varphi)$ and consider the operators $\Lambda_n\varphi = \indic_{Y}K^n(\indic_{\{R = n\}}\varphi)$. Note that we can also define it as in (\ref{Kn_def_sum}) considering only trajectories $x \in Y, x_1,\ldots,x_{n-1} \in [0,1/2]$ and $x_n \in Y$, so that we can see them as excursions of length $n-1$ outside $Y$. The operator $\Lambda_n$ is thus given for any $x \in [0,1]$ by: ~ 
\begin{align*}
    \Lambda_n \varphi (x) =\indic_{Y}(x) \frac{h(v_1 \circ v_0^{n-1}(x)) \abs{(v_1 \circ v_0^{n-1})'(x)}}{h(x)} \varphi(v_1 \circ v_0^{n-1}x).
\end{align*}
~\\[11pt]

First, let us check that the renewal equation is verified. The operator $T_n$ is linked to the operators $\Lambda_k$ by the following renewal equation:
\begin{align}\label{eq_renouv}
    T_n = \sum_{j = 1}^{n} \sum_{k_1 +\cdot \cdot \cdot +k_{j}=n}\Lambda_{k_1} \circ \cdots \circ \Lambda_{k_{j}}.
\end{align}
The density $h$ is positive and Lipschitz continuous on $[1/2,1]$ according to \parencite[Theorem 1.4.10]{GouThese}, therefore $h \circ v_1 \circ v_0^{n-1}$ and $1/h$ have bounded variation on $ Y = ]1/2,1]$. Moreover $v_1 \circ v_0^{n-1}$ is continuous and injective from $Y$ to itself, then $V(h \circ v_1 \circ v_0^{n-1})$ can be bounded independently with $n$. By submultiplicativity, we get for any $\varphi \in \mathcal{BV}$: 
\begin{align*}
    V(\Lambda_n \varphi) \leq C V\hspace{-3pt}\bigp{(v_1 \circ v_0^{n-1})'\indic_{\{R = n \}}} V(\varphi). 
\end{align*}
Therefore, it is sufficient to control the variation of $(v_1 \circ v_0^{n-1})'$. First, note that the distortion hypothesis over $f$ can be expressed from the inverse branches. By \parencite[Lemma 6.5]{Young} there exists $C > 0$ such that for any $k$, any $x,y \in J_k$ and any $n$, 
\begin{align}\label{distortion_v}
    \abs{1-\frac{(v_0^n)'(x)}{(v_0^n)'(y)}} \leq C.
\end{align}
Since $v_0^n$ is increasing, $(v_0^n)' > 0$ and there exists $C$ such that
 $(v_0^n)'(x) \leq C (v_0^n)'(y)$. By integrating with respect to $y \in J_k$,
\begin{align*}
    {(v_0^n)'(x)} \lambda(J_k) 
    &\leq C \int_{J_k} (v_0^n)'(y) dy  \\ 
    &= C (v_0^n(z_k) - v_0^n(z_{k+1}))  \\ 
    &= C (z_{n+k} - z_{n+k+1})  \\ 
    &= C \lambda(J_{n+k}). 
\end{align*}
Similarly we show that $\lambda(J_{n+k}) \leq C\lambda(J_k) (v_0^n)'(x)$. Therefore, for any $x \in J_k$,
\begin{align}\label{encadrement_v_0^n'}
    \frac{1}{C}\frac{\lambda(J_{n+k})}{\lambda(J_k)} \leq {(v_0^n)'(x)} \leq  C\frac{\lambda(J_{n+k})}{\lambda(J_k)}.
\end{align}
In particular, using the submultiplicativity of $V$, one has $V\bigp{(v_1 \circ v_0^{n-1})'\indic_{\{R = n \}}} \leq C\lambda(J_n)$ with $\lambda(J_n) = O(n^{-p-1}s(n))$ according to \parencite[Theorem 1.4.10 and Proposition 1.4.12]{GouThese}. Thus $\lVert \Lambda_n \rVert _{\mathcal L(\mathcal{BV})} = O(n^{-p-1}s(n))$. In particular, as $s(n) \to 0$, $\lVert \Lambda_n \rVert _{\mathcal L(\mathcal{BV})} = O(n^{-(p+1)})$. Moreover using slowly varying functions properties (see Lemma \ref{double_somme}), 
\begin{align*}
    \lVert T_n \rVert_{\mathcal L(\mathcal{BV})} &\leq \sum_{j=1}^n\sum_{k_1+\cdot \cdot \cdot+k_j=n}  \lVert \Lambda_{k_1}\rVert_{\mathcal L(\mathcal{BV})}\cdots \lVert \Lambda_{k_j}\rVert_{\mathcal{L}(\mathcal{BV})} \\
    &\leq  C\sum_{k=1}^n \frac{s(k)}{k^{p+1}} + C\sum_{j=2}^n j\frac{s(n)}{n^{p+1}}.
\end{align*}
We get $ \lVert T_n \rVert_{\mathcal L (\mathcal{BV})} = O(1)$. Now, let us set $\Lambda(z) = \sum_{n\geq 1}z^n\Lambda_n$ and $T(z) = \sum_{n \geq 1}z^nT_n$ that are well-defined operators on $\mathcal{BV}$, for $\abs{z} \leq 1 $ and $\abs{z} < 1$ respectively. They satisfy the first condition (\ref{eq_renouv}) of Theorem \ref{thm_renewal}:
\begin{align*}
    T(z) = (id-\Lambda(z))^{-1}.
\end{align*}

Let us now deal with the spectral gap as done in \cite{Gou2007} by applying Theorem \ref{Ruelle_thm} proved by Ruelle in \cite{Ruelle} to bound the essential spectral radius of $\Lambda(z)$. For the sake of clarity let us make the proof. For $z \in \mathbb D$, we define an operator on bounded functions by 
\begin{align*}
    \hat{\Lambda}(z)\varphi(x) = \sum_{n=1}^{\infty} z^n \Lambda_n \varphi(f^n(x)) \indic_{R(x)=n, x \in Y}.
\end{align*}
Consider $r = \underset{m \to \infty}{\lim} \norm{\hat{\Lambda}(z)^m}_{\mathcal L(\mathcal{BV})}^{1/m}$. Then, the essential spectral radius of $\Lambda(z)$ on $\mathcal{BV}$ is bounded by $r$ (see Theorem \ref{Ruelle_thm}). 
\\
Now, $\Hat{\Lambda}(z)$ can be written as $\Hat{\Lambda}(z)\varphi(x) = z^{R(x)}\frac{h(x)}{h(F_Y(x))\abs{F_Y'(x)}}\varphi(F_Y(x))\indic_{x \in Y}$ (where $R$ is the return time), so that inductively we get for any $k$
\begin{align*}
    \Hat{\Lambda}(z)^k\varphi(x) =  \frac{h(x)}{h(F_Y^k (x))}\bigp{\prod_{i=0}^{k-1}  z^{ R(F_Y^{i}(x))}\frac 1 {\abs{F_Y'(F_Y^{i}(x))}}}\varphi(F_Y^{k}(x))\indic_{x \in Y}.
\end{align*}

\comments{Following what has been done in \cite{Gou2007}, we will apply Theorem \ref{Ruelle_thm} proved by Ruelle in \cite{Ruelle} to bound the essential spectral radius of $\Lambda(z)$. For $z \in \mathbb D$, we define an operator on bounded functions by $\Hat{\Lambda}(z)\varphi(x) = \sum_{n=1}^{\infty} z^n \Lambda_n \varphi(f^n(x))$. Consider $r = \underset{m \to \infty}{\lim} \norm{\Hat{\Lambda}(z)^m}_{\mathcal L(\mathcal{BV})}^{1/m}$.\\
Then the essential spectral radius of $\Lambda(z)$ on $\mathcal{BV}$ is bounded by $r$ (see Theorem \ref{Ruelle_thm}). 
Since $\indic_{Y}(f^n(x)) = \indic_{\{R=n\}}(x)$\aurelie{cette égalité est fausse, donc autre argument à donner pour justifier l'égalité qui suit.. + si on avait $\Lambda_n (\varphi  \circ f^n)(x)$ j'intuite un peu le truc mais on a $\Lambda_n \varphi(f^n(x))$}, we have $\Hat{\Lambda}(z)\varphi(x) = z^{R(x)}\frac{h(x)}{h(F_Y(x))}\abs{F_Y'(x)}^{-1}\varphi(F_Y(x))$. We obtain, inductively,
\begin{align*}
    \Hat{\Lambda}(z)^m\varphi(x) =  \frac{h(x)}{h(F_Y^m (x))}\Bigg (\prod_{i=1}^m  z^{ R(F_Y^{i-1}(x))}\abs{F_Y'(F_Y^{i-1}(x))}^{-1}\Bigg )\varphi(F_Y^{m}(x)).
\end{align*}}
Since $h$ is continuous and positive on $Y$, $h/(h\circ F_Y^k)$ is bounded on this subset and we get $\lVert \Hat{\Lambda}(z)^m \varphi \lVert_{\infty} \leq C \norm{z^{R}/F_Y'}_\infty^m \norm{\varphi}_{\infty}$.
Consequently, for any $z \in \Bar{\mathbb{D}}$, the essential spectral radius of $\Lambda(z)$ on $\mathcal{BV}$ is bounded by $\norm{z^{R}/F_Y'}_\infty < 1$ since $F_Y$ is uniformly expanding. Moreover, $1$ is an isolated eigenvalue with finite multiplicity of $\Lambda(1)$, and since $ \lVert \Lambda(1)^n \rVert_{\mathcal L(\mathcal{BV})}$ is bounded, the operator $id - \Lambda(1)$ does not have any nilpotent part. 
\\

We will provide a detailed proof of aperiodicity following the top of page 1495 in \cite{Gou2007}.

{In order to show that $id-\Lambda(z)$ is invertible (for $z \neq 1$), it is sufficient to show that $1$ is not an eigenvalue of $\Lambda(z)$ for $\abs{z}=1, z \neq 1$. \newline 
\begin{lem}
   Let $\abs{z}=1$ and assume that there exists $\varphi \in \mathcal{BV}$, non-identically zero  and with support in $Y$ such that $\Lambda(z)\varphi = \varphi$. 
   Then $z = 1$ and $\varphi$ is constant everywhere on $Y$.
\end{lem}

\begin{proof}
Write $z = e^{it}$, where $t \in [0, 2\pi[$. Following \parencite[Lemma 6.7]{Gou2004}, we first show that $e^{-itR}\varphi \circ F_Y = \varphi$ a.e. on $Y$.
Let us denote $\langle \cdot , \cdot \rangle $ the scalar product in $L^2(\nu_Y)$ and define for any bounded function $u$, $Wu = e^{-itR}u\circ F_Y$. Note that $\Lambda(1)$ is the transfer operator associated with the transformation $F_Y$. It follows that for any bounded functions $u,v$,
\begin{align*}
    \langle u,\Lambda(z)v \rangle 
    &= \int_Y \overline{u}\Lambda(z)(v) d\nu_Y \\
    &= \int_Y \overline{u}\Lambda(1)(e^{itR}v) d\nu_Y \\ 
    &= \int_Y \overline{u}\circ F_Y(e^{itR}v) d\nu_Y \\
    &= \int_Y \overline{e^{-itR}u} \, v d\nu_Y \\
    &= \langle Wu,v \rangle.
\end{align*}
We get $\langle W\varphi,\varphi\rangle = \langle \varphi, \Lambda(z)\varphi \rangle =  \langle \varphi,  \varphi
\rangle$. Therefore
\begin{align*}
    \lVert W\varphi -\varphi\rVert_2^2 
    = \lVert W\varphi \rVert_2^2 + \lVert \varphi\rVert_2^2 - 2 \text{Re}\langle W\varphi,\varphi\rangle
    =  \lVert W\varphi \rVert_2^2 -  \lVert \varphi\rVert_2^2.
\end{align*}
 Since $\nu_Y$ is $F_Y$-invariant, we have $\lVert W\varphi \rVert_2^2 =  \lVert \varphi\rVert_2^2$ and $W\varphi = \varphi$ that is
\begin{align}\label{eq_periodicite}
    e^{-itR}\varphi \circ F_Y = \varphi
\end{align}
almost everywhere on $Y$.
Let $Y_0 $ be a full-measure subset of $Y$ on which $\varphi$ satisfies (\ref{eq_periodicite}) everywhere. The function $F_Y$ induces a bi-measurable bijection between each interval $\{R = n \}$ and $Y$. Set $Y_1 = Y_0\cap_{n \geq 1}F_Y(Y_0\cap \{R = n \})$ that satisfies then $\nu_Y(Y_1) = 1$.
Let us show that $\varphi$ is constant on $Y_1$. Consider $x\neq y \in Y_1$. For any $n \geq 1$, there exist $x_n, y_n \in \{ R = n \} \cap Y_0 $ such that $ F_Y(x_n) = x$ and $F_Y(y_n) = y$. By construction, the sequences $(x_n)_n,(y_n)_n$ are decreasing, then for any $n \geq 1$:
\begin{align*}
    V(\varphi) 
    \geq \sum_{k=1}^n \abs{\varphi(x_k)-\varphi(y_k)} 
    = \sum_{k=1}^{n} \abs{e^{-itk}\varphi(x) - e^{-itk}\varphi(y)}
    = n\abs{\varphi(x)-\varphi(y)}.
 \end{align*}
 Since $V(\varphi) < \infty$, we infer that $\varphi$ is constant a.e. on $Y$ equal to some $c$. On the other hand, $\{ R = 1 \}$ is non-empty then (\ref{eq_periodicite}) implies $e^{-it} = 1$, so that $z =1$. \\ 
Consider $\psi = \varphi-c\indic_Y$
that verifies $\Lambda(1)\psi = \psi$ according to what precedes and let us show that $\psi = 0$ everywhere on $Y$. 
Assume that $\psi(x) \neq 0$ for some $x \in Y$, one has
\begin{align*}
    \psi(x) = \Lambda(1)\psi(x) = \frac{1}{h(x)}\sum_{f(y) = x} \frac{h(y)}{\abs{f'(y)}}\psi(y).
\end{align*}
Moreover, since $\Lambda(1)$ is the transfer operator associated with the induced transformation on $Y$, $\Lambda(1)\indic_{Y} = \indic_{Y}$. Hence 
\begin{align*}
1 = \frac{1}{h(x)}\sum_{f(y) = x}\frac{h(y)}{\abs{f'(y)}}  = \frac{1}{h(x)}\sum_{f(y) = x}\frac{h(y)}{\abs{f'(y)}}\frac{\psi(y)}{\psi(x)}.
\end{align*}
Since $\{ \psi = 0 \}$ is dense in $Y$, if $ \psi(y) =\psi(x)$ for any antecedent $y$ of $x$ by $f$ then we would have $V(\psi) \geq n\psi(x)$ for any arbitrary $n$. $V(\varphi)$ would then be infinite, which is impossible. 
Therefore there exists some $x_1 \in Y$ such that $f(x_1) = x$ and $\abs{\psi(x_1)} > \abs{\psi(x)}$. We then repeat the reasoning with $x_1$ instead of $x$. If the procedure does not stop we have an infinity of $x_n$ such that $\abs{\psi(x_n)} > \abs{\psi(x)}$, and thus $V(\psi) = \infty$, which is absurd. It must end, but then we also obtain as above that $V(\psi)$ is infinite. Finally $\varphi = c\indic_{Y}$ what achieves the proof of the lemma.
\end{proof}
According to the previous lemma, for $z \neq 1$, $1$ is not an eigenvalue of $\Lambda(z)$ and the essential spectral radius of $\Lambda(z)$ is lower than $1$. Thus $id-\Lambda(z)$ is invertible as expected. 
Furthermore, the characteristic space of $\Lambda(1)$ associated to $1$ is $\mathbb{C}\indic_Y$ and, since $id - \Lambda(1)$ does not have any nilpotent part, $1$ is a simple eigenvalue.
Now, we follow the computations given in \cite{Gou2007} page 1495 but with the reference measure $\nu_Y$ instead of the Lebesgue measure. Hence, the projection $P$ of $\Lambda(1)$ associated to $1$ is given by
 \begin{align*}
     P\varphi = \Big (\int_Y \varphi  d\nu_Y \Big)\indic_Y
 \end{align*}
and we can show that $P\Lambda'(1)P =  \frac{1}{\nu(Y)}P$. Indeed, $P\Lambda'(1)P\varphi = \Big(\int_Y \Lambda'(1) \indic_Y d\nu_Y\Big)P\varphi$ and:
 \begin{align*}
     \int_Y \Lambda'(1)\indic_{Y} d\nu_Y 
     &=\frac{1}{\nu(Y)}\int_Y \sum_{n=1}^{\infty} n \Lambda_n(\indic_Y) d\nu \\ 
     &=  \frac{1}{\nu(Y)}\sum_{n=1}^{\infty}n \int_Y\indic_Y K^n(\indic_{\{ R = n\}}) d\nu \\ 
     &= \frac{1}{\nu(Y)}\sum_{n=1}^{\infty}n \nu(\{R=n \})
     \\ &= \frac{1}{\nu(Y)}
 \end{align*}
 where the last equality follows from Kac's formula (see for instance Proposition 1.3.4 in \parencite[]{GouThese}).
}
\\

According to what precedes, all the hypotheses of Gouëzel's Theorem \ref{thm_renewal} are verified with $\mathbb B = \mathcal{BV}$, the operators $T_n$, $\mu = 1/\nu(Y)$ and $q = p$. We get for any $n \geq 1$:
 \begin{align*}
     T_n = \nu(Y)Q + \Tilde{E_n} + N
 \end{align*}
 where $\Tilde{E_n} = {\nu(Y)^2}\sum_{j = n+1}^\infty (j-(n+1)) Q\Lambda_jQ $ and $\norm{N}_{\mathcal L(\mathcal{BV})} = O(n^{-p})$. 
 Moreover, one has $\lVert \Lambda_j \rVert_{\mathcal L(\mathcal{BV})} = O(j^{-p-1}s(j))$, hence
 \begin{align*}
     \lVert \Tilde{E_n} \rVert_{\mathcal L(\mathcal{BV})} 
     \leq {\nu(Y)^2}\sum_{j = n+1}^\infty (j-(n+1)) \lVert Q\rVert_{\mathcal L(\mathcal{BV})}^2 \lVert \Lambda_j \rVert_{\mathcal L(\mathcal{BV})} 
     \leq C\sum_{j = n+1}^\infty (j-(n+1)) j^{-p-1}s(j).
 \end{align*}
Now $\sum_{j = n+1}^\infty j^{-p}s(j) = O(n^{-p+1}s(n))$, which concludes the proof of Proposition \ref{decomp_Tn}. 

\subsection{Study of the operators \texorpdfstring{$A_n, B_n, C_n$}{}}
In order to study $A_n, B_n$ and $C_n$, let us adapt what has been done in \parencite[Section 3]{DGM2010} to our context. It is worth noting that in our case, $(v_0^n)'$ is not necessarily decreasing so that we will have to adapt their proof. First, let us show an analogue of \parencite[Lemma 3.4]{DGM2010} following their arguments.
\begin{lem}
    There exists a constant $C>0$ such that for any $1 \leq i < j$
    \begin{align*}
        V(\indic_{[z_j,z_i]}h) \leq Cj \text{ and }  V(\indic_{[z_j,z_i]}/h) \leq Cj/i^2.
    \end{align*}
\end{lem}

\begin{proof}
Combining the properties of the bounded variation $V$ with (\ref{encadrement_h}), note that 
\begin{align}
        \label{proof_decomp_Tn:control_V}
    V(\indic_{[z_M,z_N]}h)
    &\leq 
    C\sum_{n = 0}^\infty \bigp{\sup_{[z_M,z_N]}{(v_0^n)'} + 3 \sum_{k = N}^{M-1} \norm{d\bigp{\indic_{J_k}(v_0^n)'}}}
    \nonumber \\
    &\leq
    C\sum_{n = 0}^\infty \bigp{\sup_{k = N \cdots M} \frac{\lambda(J_{n+k})}{\lambda(J_{k})} + 3 \sum_{k = N}^{M-1} \norm{d\bigp{\indic_{J_k}(v_0^n)'}}}
    \nonumber \\
    &\leq
    C\sum_{n = 0}^\infty \bigp{\sup_{k = N \cdots M} \frac{s(n+k) k^{p+1}}{s(k) (n+k)^{p+1}} + 3 \sum_{k = N}^{M -1} \norm{d\bigp{\indic_{J_k}(v_0^n)'}}}.
\end{align}

Taking into account Lemma \ref{Potter_thm}, there exists $k_0 \in \mathbb N$ such that for any $k \geq k_0$ and any $n$, $\frac{s(n+k)}{s(k)}\leq 2\bigp{\frac{n+k}{k}}^{p-1/2}$. Noting that $\left\{\bigp{\frac{k}{n+k}}^{3/2}\right\}_{k \geq 1}$ is increasing and $\sum_{n = 0}^\infty \bigp{\frac{j}{n+j}}^{3/2}$ is of order $j$, we infer that 
\begin{align}
        \label{proof_decomp_Tn:control_sup_ell}
    \sum_{n = 0}^\infty \sup_{k = k_0 \cdots j} \frac{s(n+k) k^{p+1}}{s(k) (n+k)^{p+1}} \leq Cj.
\end{align}
Now, combining (\ref{distortion_v}) and (\ref{encadrement_v_0^n'}), we infer that $(v_0^n)'$ is Lipschitz on each interval $J_k$ with a Lipschitz constant $l_k \leq C\frac{\lambda(J_{n+k})}{\lambda(J_k)}$, so that $\norm{d\bigp{\indic_{J_k}(v_0^n)'}} \leq C \lambda(J_{n+k})$. 
Hence, taking into account (\ref{proof_decomp_Tn:control_V}), (\ref{proof_decomp_Tn:control_sup_ell}) and  (\ref{estim_zn}), 
\begin{align*}
     V(\indic_{[z_j,z_i]}h) 
     \leq 
     V(\indic_{[z_j,z_{k_0}]}h) + V(\indic_{[z_{k_0}, z_i]}h) 
     \leq 
     Cj + C
     \leq 
     Cj. 
\end{align*}
(The case when $k_0 \notin \{i, i+1 \cdots, j\}$ follows since $h \in \mathcal{BV}$ so that $V(\indic_{[\varepsilon, 1]}h)$ is bounded for any $\varepsilon > 0$.)
Finally, notice that $\abs{\frac{1}{h(x)}-\frac{1}{h(y)}} = \abs{\frac{h(y)-h(x)}{h(x)h(y)}}$ so that 
\begin{align*}
    V(\indic_{[z_j,z_i]}/h) \leq \frac{V(\indic_{[z_j,z_i]}h)}{(\inf_{[z_j,z_i]}h)^2} \leq Cj/i^2.
\end{align*}
\end{proof}

In a similar way, we can adapt the proofs of Lemma 3.5 and Lemma 3.6 in \parencite[]{DGM2010} so that following their arguments we can state that for any $\varphi \in \mathcal{BV}$: 
\begin{align}
    &V(C_n\varphi)\leq CV(\varphi) \label{bound_Cn}\\ 
    &V(A_n\varphi) \leq CV(\varphi)/(n+1) \label{bound_An}\\ 
    &V(B_n\varphi) \leq CV(\varphi)s(n)/(n+1)^p. \label{bound_Bn}
\end{align}

We can now state an analogous result to \parencite[Proposition 1.15]{DGM2010} to establish the uniform boundedness of the iterates of $K$.
\begin{prp}\label{analogue1.15}
    There exists $C > 0$ such that for any $n \geq 1$ and any $\varphi \in \mathcal{BV}$,
   $\lVert dK^n\varphi\rVert \leq C\norm{d\varphi}$. 
\end{prp}

\begin{proof}
As a first step, assume that $\nu(\varphi) = 0$.
For any $n \geq 1$, according to Proposition \ref{decomp_Tn}, (\ref{bound_An}) and (\ref{bound_Bn}), 
\begin{align*}
    V \Big( \sum_{a+k+b=n} A_a \circ E_k \circ B_b \varphi \Big) \leq CV(\varphi) \sum_{a+k+b = n} \frac{s(k) s(b)}{(a+1)(k+1)^{p-1}(b+1)^p}.
\end{align*}
Applying Lemma \ref{double_somme}, $\sum_{k+b = j}\frac{s(k)s(b)}{(k+1)^{p-1}(b+1)^{p}} \leq C\frac{s(j)}{(j+1)^{p-1}}$. Therefore the same computations as in the proof of Proposition 1.15 in \cite{DGM2010} associated with Lemma \ref{subpolynomial} lead to 
\begin{align}
        \label{prp1.9:V_Ek}
    V \Big( \sum_{a+k+b=n} A_a \circ E_k \circ B_b \varphi\Big) \leq CV(\varphi). 
\end{align}
Now, $\sum_{b=0}^{\infty}\nu(B_b\varphi) = \nu(\varphi) = 0$ so that
    \begin{align}
            \label{prp1.9:V_Q_B}
        \abs{\sum_{b=0}^{n-a}\nu(B_b\varphi)} 
        = \abs{\sum_{b > n-a}\nu(B_b\varphi)} 
        \leq \sum_{b > n-a} V(B_b\varphi) 
        \leq CV(\varphi) \sum_{b > n-a} \frac{s(b)}{(b+1)^p}.
    \end{align}
Hence, according to Lemma \ref{subpolynomial},
\begin{align}
        \label{prp1.9:V_Q}
    V\bigp{\sum_{a+k+b = n}A_a \circ Q \circ B_b \varphi} 
    &= 
    V\bigp{\sum_{a+k+b = n}A_a(\indic_Y) \nu(B_b \varphi)}\nonumber
    \\ &\leq 
    C V(\varphi)\sum_{i = 0}^n \frac{s(n+1-i)}{(i+1)(n+1-i)^p}\nonumber
    \\ & \leq 
    CV(\varphi).
\end{align}
Thus, as soon as $\nu(\varphi) =0$, combining the decomposition provided by (\ref{itérées_Kn}) and Proposition \ref{decomp_Tn} with the estimates (\ref{prp1.9:V_Ek}), (\ref{prp1.9:V_Q}) and (\ref{bound_Cn}) we infer that $V(K^n\varphi) \leq CV(\varphi)$.
Finally, for any $\varphi \in \mathcal{BV}$, the result follows since
\begin{align*}
    \norm{dK^n\varphi}
    =\norm{dK^n\varphi^{(0)}} 
    \leq V(K^n\varphi ^{(0)}) 
    \leq CV(\varphi^{(0)}) 
    \leq 3C\norm{d\varphi}. 
\end{align*}
\end{proof}

On another hand, in order to estimate weak $\alpha$-mixing coefficients, we can state the following covariance inequality, analogous to \parencite[Proposition 1.16]{DGM2010}:
\begin{prp}\label{analogue1.16}
    There exists $C > 0$ such that for any $n \geq 1$, any bounded function $\psi$ and any $\varphi \in \mathcal{BV}$:
    \begin{align*}
        \abs{\nu\bigp{\psi\circ T^n . (\varphi- \nu(\varphi))}} \leq \frac{Cs(n)}{n^{p-1}} \rVert d\varphi \lVert \rVert \psi \lVert_{\infty}.
    \end{align*}
\end{prp}
\begin{proof}
    We can assume without loss of generality that $\nu(\varphi) = 0$. We have 
    \begin{align*}
        \abs{\nu(\psi\circ T^n \varphi})
        \leq \abs{\nu(\psi K^n\varphi)}
        \leq \rVert \psi \lVert_{\infty} \nu(\abs{K^n\varphi}).
    \end{align*}
    We can estimate $\nu (\abs{K^n\varphi})$ combining the decomposition $K^n= \sum_{a+k+b}A_a \circ (E_k + Q) \circ B_b +C_n$, Proposition \ref{analogue1.15} and bounds (\ref{bound_Cn}), (\ref{bound_An}), (\ref{bound_Bn}). \\
    First
    \begin{align*}
        \nu(\abs{C_n\varphi}) &\leq C \rVert \varphi \lVert_{\infty} \nu(K^n \indic_{[0,z_{n+1}]})  
        = C \rVert \varphi \lVert_{\infty} \nu({[0,z_{n+1}]}) 
        \leq C\norm{\varphi}_{\infty}\frac{s(n)}{(n+1)^{p-1}}.
    \end{align*}
    Observe that by definition of $A_n$, 
    \begin{align*}
         \nu(\abs{A_n\psi}) 
         \leq C \norm{\psi}_{\infty} \nu (K^n\indic_{Y\cap T^{-1}([0,z_n])}) 
         = C\norm{\psi}_{\infty} \nu(Y\cap T^{-1}([0,z_n])).
    \end{align*}
    Since $h$ is bounded on $Y$, $\nu(\abs{A_n\psi}) \leq C \norm{\psi}_{\infty}z_n$ so that
    \begin{align*}
         \nu(\abs{A_n\psi}) \leq \frac{C\norm{\psi}_{\infty}s(n)}{(n+1)^p}.
    \end{align*}
    Hence
    \begin{align*}
        \nu \bigp{\abs{\sum_{a+k+b = n}A_a \circ E_k \circ B_b\varphi}} 
        &\leq CV(\varphi)\sum_{a+k+b=n}\frac{\lVert E_k \circ B_b\varphi\rVert_{\infty}s(a)}{(a+1)^p} \\
        &\leq CV(\varphi)\sum_{a+k+b=n}\frac{s(a)s(k)s(b)}{(a+1)^p(k+1)^{p-1}(b+1)^p} \\
        &\leq CV(\varphi)\sum_{k=0}^n \frac{s(k)s(n-k)}{(k+1)^{p-1}(n-k+1)^p} \\
        &\leq \frac{CV(\varphi)s(n)}{(n+1)^{p-1}}
    \end{align*}
    by using twice Lemma \ref{double_somme}. It remains to control $\sum_{a+b \leq n}A_a \circ Q \circ B_b\varphi = \sum_{a+b\leq n}A_a(\indic_Y)\nu(B_b\varphi)$. 
    As $s$ is slowly varying, according to Lemma \ref{Potter_thm}, there exists $C >0$ not depending on the indices such that $\frac{s(b)}{s(n-a)} \leq \frac{Cb}{n-a}$ for $b > n-a$ so that
    \begin{align*}
         \sum_{b > n-a} \frac{CV(\varphi)s(b)}{(b+1)^p} 
         \leq \frac{CV(\varphi)s(n-a)}{n-a}\sum_{b>n-a}b^{1-p}.
    \end{align*}
    Hence, by (\ref{prp1.9:V_Q_B}), $\abs{\sum_{b=0}^{n-a}\nu(B_bf)} \leq \frac{CV(\varphi)s(n-a)}{(n-a+1)^{p-1}}$ and
    \begin{align*}
       \nu \Big(\sum_{a+b\leq n}A_a(\indic_Y)\nu(B_b\varphi) \Big) 
       \leq CV(\varphi)\sum_{a=0}^n \frac{s(a)s(n-a)}{(a+1)^{p}(n-a+1)^{p-1}} 
       \leq \frac{CV(\varphi)s(n)}{(n+1)^{p-1}}
    \end{align*}
    applying Lemma \ref{double_somme}. Gathering the terms, we infer that $\nu(\abs{K^n\varphi}) \leq \frac{CV(\varphi)s(n)}{(n+1)^{p-1}}$ and the result follows.
\end{proof}
\subsection{Proof of the main theorems}
We can now establish estimates on the mixing coefficients.
\begin{proof}[Proof of Theorem \ref{thm_alpha}]
    Let $\varphi, \psi \in \mathcal{BV}$ such that $\norm{d\varphi} \leq 1$ and $\norm{\psi}_{\infty} \leq 1$. Recall that $ \lVert \varphi^{(0)}\rVert_\infty \leq V(\varphi^{(0)}) \leq 3\norm{d\varphi}$. By applying Proposition \ref{analogue1.15},
    \begin{align*}
        \lVert d\big( \varphi^{(0)}K^n(\psi)\big)\rVert 
        &\leq \norm{d\varphi} \norm{\psi}_{\infty} + \norm{dK^n\psi} \norm{\varphi^{(0)}}_{\infty} 
        \leq 1 + 3C\norm{d\psi}.
    \end{align*}
    Let $\varphi_1,\ldots,\varphi_k \in \mathcal{BV}_1$, $n_1,\ldots,n_k \in \mathbb{N}$ and set $\varphi = \varphi_1^{(0)}K^{n_2}(\varphi_2^{(0)}K^{n_3}(\ldots K^{n_{k-1}}(\varphi_{k-1}^{(0)}K^{n_k}(\varphi_k^{(0)}))\ldots))$. By iterating the above inequality,
    \begin{align*}
        \norm{d\varphi} \leq 1 + C + C^2 +\cdot \cdot \cdot+ C^{k-1}.
    \end{align*}
    Therefore $ \alpha_{k, \mathbf Y}(n) \leq (1+ C+\cdot \cdot \cdot+ C^{k-1}) \alpha_{1, \mathbf Y}(n)$. Now, by duality, according to Proposition \ref{analogue1.16} $\alpha_{1, \mathbf Y}(n) \leq \frac{Cs(n)}{n^{p-1}}$, what concludes the proof of the theorem.
\end{proof}
\begin{proof}[Proof of Theorem \ref{hip_BV}]
    Since $(Y_n,Y_{n-1}\ldots,Y_1) = (f, f^2,\ldots, f^n)$ in law for a Markov chain $(Y_i)_{i \in \mathbb{Z}}$ with transition kernel $K$ and invariant measure $\nu$ (see for instance \parencite[Lemma XI.3]{HH2001}), to prove Theorem \ref{hip_BV} it suffices to prove the same result with $Y_k$ replacing $\varphi \circ f^k$. This follows by applying \parencite[Theorem 3.3]{DM2017} together with the fact that by Theorem \ref{thm_alpha}, $n^{p-1}\alpha_{2, \mathbf Y}(n) = O(s(n))$ with $s(n) \longrightarrow 0$.
\end{proof}

\appendix
\section{Annex}
\subsection{Slowly varying functions}
    \label{annex:slowlyvar}
For the reader's convenience, we recall in this section some results on slowly varying functions. We refer to \parencite[Chapter 1]{BGT} for more details.
\begin{df}
    A function $\rho : \mathbb{R}_{+}^{*} \longrightarrow \mathbb{R}_{+}^{*}$ is slowly varying (at $+\infty$) if for any $\lambda > 0$ 
    \begin{align*}
        \frac{\rho(\lambda x)}{\rho(x)} \underset{x \to \infty}{\longrightarrow} 1.
    \end{align*}
    Similarly $\rho$ is slowly varying at $0$ if $x \mapsto \rho(1/x)$ is slowly varying at $+ \infty$. A sequence $(\ell(n))_n$ is slowly varying (at $+ \infty$) if there exists some slowly varying function $\rho$ such that $\ell(n) = \rho(n)$. 
\end{df}
In the definition of a Holland map there is some regularity conditions on the slowly varying function. As pointed out in Gouëzel thesis \cite{GouThese}, they are useful to control the distortion and are not too restrictive according to the following theorem.
\begin{thm}[Smooth variation theorem, \cite{BGT}]
    Let $\rho$ be a slowly varying function. There exist $\Tilde{\rho}_1$, $\Tilde{\rho}_2$ slowly varying such that
    \begin{itemize}
        \item for any $n \geq 1$ and any $i= 1,2$, $x^n{\Tilde{\rho}_i(x)}\bigp{\Tilde{\rho}_i^{(n)}(x)}^{-1} \underset{x \to \infty}{\longrightarrow} 0$,
        \item $\Tilde{\rho}_1 \sim \Tilde{\rho}_2$ and $\Tilde{\rho}_1 \leq \rho  \leq \Tilde{\rho}_2$ on a neighborhood of $+\infty$.
    \end{itemize}
\end{thm}

\begin{lem}[Proposition 1.3.6, \cite{BGT}]\label{subpolynomial}
    If $\ell$ is slowly varying (at $\infty$) then for any $a >0$
    \begin{align*}
        x^{-a}\ell(x) \underset{x \to \infty}{\longrightarrow} 0.
    \end{align*}
    In particular if $a > 1$ then $\sum_{n=1}^{\infty}\ell(n)n^{-a} < \infty$.
\end{lem}

\begin{lem}[Theorem 1.5.6 (i), \cite{BGT}]\label{Potter_thm}
    If $\ell$ is slowly varying then for any $A >1$ and $\delta >0$ there exists $x_0=x_0(A,\delta)$ such that for any $x_0 \leq x \leq y$
    \begin{align*}
        \frac{1}{A}\Big(\frac{y}{x}\Big)^{-\delta} \leq \frac{\ell(y)}{\ell(x)} \leq A\Big(\frac{y}{x}\Big)^{\delta}.
    \end{align*}
\end{lem}

\begin{lem}\label{double_somme}
        Let $u,v$ be positive slowly varying sequences and let $s,r > 1$, then
        \begin{align*}
            \sum_{i+j=n} \frac{u_iv_j}{(i+1)^r(j+1)^s} \leq C \Big(\frac{u_n}{(n+1)^r}+ \frac{v_n}{(n+1)^s} \Big).
        \end{align*}
    \end{lem}
    \begin{proof}
       We split the sum into two parts:
       \begin{align*}
            \sum_{i+j=n} \frac{u_iv_j}{(i+1)^r(j+1)^s}
            &\leq \sum_{n/2 \leq i \leq n} \frac{u_iv_{n-i+1}}{(i+1)^r(n-i+1)^s} + \sum_{n/2 \leq j \leq n}\frac{u_{n-j+1} v_j}{(n-j+1)^r(j+1)^s} \\
            &\leq 2^rn^{-r}u_n \sup_{n/2 \leq i \leq n} \frac{u_i}{u_n} \sum_{k \geq 1}v_k k^{-s} +  2^sn^{-s}v_n \sup_{n/2 \leq j \leq n} \frac{v_j}{v_n} \sum_{k \geq 1} u_kk^{-r}.
       \end{align*}
       Since $r,s > 1$, both above sums are finite. Moreover, the suprema are bounded independently of $n$ by Lemma \ref{Potter_thm}, which provides us the expected upper bound.
    \end{proof}
    
\subsection{Ruelle's theorem for generalized transfer operator}
We present here Ruelle's theorem in a simplified form that we used to control the essential spectral radius of $\Lambda(z)$. 
Consider for each $n \geq 1$:
\begin{itemize}
    \item $I_n$ a sub-interval of $[0,1]$,
    \item a continuous and increasing function $\psi_n : I_n \longrightarrow [0,1] $,
    \item $\phi_n : [0,1] \longrightarrow \mathbb{C}$ of bounded variation and such that $\sum_{n=1}^{\infty}V(\phi_n) < \infty$. 
\end{itemize}
Consider the operators defined on $\mathcal{BV}$ by:
\begin{align*}
    &\mathcal{M}f(x) = \sum_{n=1}^{\infty}\phi_n(x)f(\psi_n(x))
    \qquad \text{and} \qquad
    \Hat{\mathcal{M}}f(x) = \sum_{n=1}^{\infty}\phi_n(\psi_n^{-1}(x))f(\psi_n^{-1}(x))
\end{align*}
with $f(\psi_n(x)) =0$ if $x \notin I_n$. 
With these notations in mind, point (b) of Theorem B.1 in \cite{Ruelle} writes as follows.
\begin{thm}[Ruelle]\label{Ruelle_thm}
    The essential spectral radius of $\mathcal{M}$ is bounded by $s = \underset{m \to \infty}{\lim}(\lVert \Hat{\mathcal{M}}^m)\rVert_{\mathcal{L}(\mathcal{BV})})^{1/m}$.
\end{thm}
    ~\\
    
    \noindent\textbf{Acknowledgements. }We would like to thank C. Cuny, J. Dedecker and F. Merlevède for helpful discussions. We also thank the reviewer for his careful reading of the paper and his comments which improved the presentation of the paper. The first author would like to thank both MAP5 (Université Paris Cité) and LMBA (Université de Bretagne Occidentale) laboratories for hosting him during his internship, during which a part of this article was written.
    
    \printbibliography
\end{document}